\documentclass[12pt]{amsart}
\usepackage{amscd}
\newtheorem{thm}{Theorem}[section]
\newtheorem{lem}[thm]{Lemma}
\newtheorem{cor}[thm]{Corollary}

\newtheorem{fact}{Fact}
\theoremstyle{remark}
\newtheorem{rem}[thm]{Remark}

\def\b{\mathbb B}

\def\adp{{{\mathbb A}^{\infty,p}}}
\def\c{\mathbb C}
\def\f{\mathbb F}
\def\h{\mathbb H}
\def\n{\mathbb N}

\def\q{\mathbb Q}
\def\r{\mathbb R}
\def\z{\mathbb Z}
\def\d{\mathbb D}

\def\A{\mathcal A}
\def\BT{\mathcal{BT}}

\def\O{\mathcal O}

\def\S{\mathcal S}
\def\gm{\mathfrak{m}}
\def\gh{\mathfrak{h}}

\def\Abar{\overline A}

\def\Mbar{\overline M}

\def\Xbar{\overline X}
\def\Ybar{\overline Y}
\def\ebar{\overline\eta}

\def\id{\operatorname{id}}

\def\GL{\operatorname{GL}}
\def\Spc{\operatorname{Spec}}

\def\Prj{\operatorname{Proj}}
\def\Mat{\operatorname{Mat}}
\def\lg{\operatorname{length}}
\def\End{\operatorname{End}}
\def\Hom{\operatorname{Hom}}
\def\i{\operatorname{i}}
\def\rk{\operatorname{rank}}
\def\cds{\operatorname{crisdisc}}
\def\tr{\operatorname{tr}}

\title{On the supersingular loci of quaternionic Siegel space}

\author{Oliver B\"ultel}

\thanks{Mathematische Fakult\"at der Universit\"at Duisburg-Essen, 
Universit\"atsstrasse 2, 45117 Essen, Germany, oliver.bueltel@uni-due.de, 
Subject Classification(2000): 14L05, 14K10}

\begin{document}
\maketitle

\begin{abstract}
The paper studies the supersingular locus of the characteristic $p$ moduli 
space of principally polarized abelian $8$-folds that are equipped with an 
action of a maximal order in a quaternion algebra, that is non-split at 
$\infty$ but split at $p$. The main result is that its irreducible components 
are Fermat surfaces of degree $p+1$.
\end{abstract}

\tableofcontents

\section{Introduction}

\label{intro}

Let $p$ be a prime number. In ~\cite{lioort} Oort and Li give a description of the supersingular locus 
$\S_{g,1}$ of $\A_{g,1}\times\f_p$, the fibre over $p$ of the Siegel modular variety of principally polarized 
abelian $g$-folds. Among their results are that $\S_{g,1}$ has $H_g(p,1)$ irreducible components if $g$ is odd 
and $H_g(1,p)$ if $g$ is even, and all of these components have dimension $[g^2/4]$.\\
In this paper we study the supersingular locus of certain $PEL$-moduli spaces $S_K$ of type $D_4^\h$, see body 
of text for a more precise explanation. These moduli spaces are associated to groups $G$ that are twists of 
$GO(8)$. In the complex analytic context there exist uniformisations by quaternionic Siegel half-spaces, these 
are tube domains of the shape

\begin{equation}
\label{siegel}
\gh=\{X+iY|X,Y\in\Mat_2(\h), X^{t,\iota}=X, Y^{t,\iota}=Y>0\},
\end{equation}

where $\h$ is the non-split quaternion algebra over $\r$, and $\iota$ is the standard involution.\\
In the algebraic context $S_K$ is a $6$-dimensional variety parameterizing abelian $8$-folds with a particular 
kind of additional structure, and on a mild assumption on the level structure
this variety is smooth. For every prime of good reduction we introduce
the usual integral model for this Shimura variety, and we move on to exhibit the geometry
of each individual irreducible component of the supersingular locus of the
$\mod p$-reduction. Our main result says that these components are Fermat
surfaces. This comes as a surprise, because for a more general Shimura
variety, the structure of the supersingular locus is usually quite complicated and might not even be smooth,
for example this happens in the case of $\S_{3,1}$ the $2$-dimensional space
of supersingular principally polarized abelian $3$-folds, cf. ~\cite[Paragraph(4)]{odaoort}, 
~\cite[Proposition 2.4]{oorta}, ~\cite[Example(9.4)]{lioort} or ~\cite{richartz} for a very precise exposition. 
We round off the discussion by turning to the non-supersingular points also, we prove that their $p$-divisible 
groups do not have parameters, which is somewhat the exact opposite to their behaviour on the supersingular locus. 
This too seems unusual, as one sees from the non-supersingular principally polarized abelian $3$-folds. These results 
were already applied in ~\cite{bu} to obtain an Eichler-Shimura congruence relation for $S_K$ and for its Shimura 
divisors.\\
There remains the pleasant task of thanking C.Kaiser, Prof. M.Rapoport, Prof. T.Wedhorn for remarks on the topic 
and especially Prof. R.Taylor for much good advice and Prof.F.Oort for some email exchanges.\\

The paper is organized as follows: Section ~\ref{struc} focuses on local aspects, section ~\ref{datum} on 
global ones. In subsections ~\ref{notat}/~\ref{yoga} we sum up definitions and conventions. In subsections 
~\ref{locpre}  we explain techniques needed to understand supersingular Dieudonn\'e 
modules. We apply these techniques to supersingular Dieudonn\'e modules with 
the particular additional structure under consideration in the subsections 
~\ref{Alocres}-~\ref{Blocres}. In subsection \ref{fibre} a result on the 
non-supersingular locus is obtained (Corollary \ref{blow}).

\section{Structure Theorems on Dieudonn\'e Modules with Pairing}

\label{struc}
\subsection{Notions and Notations}
\label{notat}

We continue to fix a prime $p$. If $k$ is a perfect field of characteristic $p$, then one denotes by $W(k)$ and $K(k)$ 
the Witt ring and fraction field thereof. Unless otherwise said, $k$ will be assumed to be algebraically closed. The 
absolute Frobenius $x\mapsto x^p$ induces automorphisms $x\mapsto{^Fx}$ on $W(k)$ and $K(k)$ which again will be referred to 
as absolute Frobenii. Recall that a Dieudonn\'e module is a finitely generated, torsion free $W(k)$-module 
together with a $^F$-linear endomorphism $F$ and a $^{F^{-1}}$-linear endomorphism $V$ that satisfies $FV=VF=p$. 
If one tensorizes with $\q$ one obtains the isocrystal of $M$, this is a finite dimensional $K(k)$-vector space 
together with a $^F$-linear bijection $F$. Dieudonn\'e modules are called isogenous if they give rise to isomorphic 
isocrystals.\\

By a pairing on a Dieudonn\'e module $M$ one understands a $W(k)$-bilinear map $\phi:M\times M\rightarrow W(k)$ which satisfies
$\phi(x,Fy)={^F\phi(Vx,y)}$. When thinking of $M$ as the co-variant Dieudonn\'e module of a $p$-divisible group $A$ over $k$, this means that 
$\phi$ gives rise to a morphism from $A$ to the Serre-dual of $A$. Dieudonn\'e modules with pairings $(M,\phi)$ and $(M',\phi')$ are called 
isometric if there exists an isomorphism from $M$ to $M'$ taking $\phi$ to $\phi'$. The dimension, $\dim_k(M/VM)$ of a Dieudonn\'e module with 
non-degenerate pairing is equal to the codimension $\dim_k(M/FM)$, the rank of $M$ is necessarily even. A pairing is called antisymmetric if 
$\phi(x,y)=-\phi(y,x)$ and symmetric if $\phi(x,y)=\phi(y,x)$. In either of these cases we denote
$\{x\in M\otimes\q|\phi(M,x)\subset W(k)\}$ by $M^t$, if $M=M^t$ we say that $\phi$ is perfect.\\
In \cite[Definition(3.1)]{ogus} the crucial notion of crystalline discriminant of a non-degenerate symmetric pairing is introduced: Say the 
underlying Dieudonn\'e module has rank $2n$, and choose a $K(k)$-basis $x_1,\dots,x_{2n}$ with the additional property 
$Fx_1\wedge\dots\wedge Fx_{2n}=p^nx_1\wedge\dots\wedge x_{2n}$. The determinant $\det(\phi(x_i,x_j))$, regarded as an element in
$K(\f_p)^\times/(K(\f_p)^\times)^2$, can be checked to be independent of the choice of basis and is called the crystalline discriminant $\cds(M,\phi)$. It 
only depends on the isogeny class of $M$ allowing one to also write $\cds(M\otimes\q,\phi)$ for $\cds(M,\phi)$. When fixing once and for all an 
element $t\in W(\f_{p^2})^\times$ with $t^\sigma=-t$, the target group $K(\f_p)^\times/(K(\f_p)^\times)^2$ can be identified with $\{1,t^2,p,pt^2\}$ 
if $p$ is odd and with $\{\pm1,\pm t^2,\pm p,\pm pt^2\}$ if $p=2$, notice that the kernel of the forgetful map from $K(\f_p)^\times/(K(\f_p)^\times)^2$ to
$K(k)^\times/(K(k)^\times)^2$ consists of $\{1,t^2\}$. Notice also that the image of $\cds(M,\phi)$ in the group $K(k)^\times/(K(k)^\times)^2$ is the 
discriminant of a symmetric pairing in the usual sense of linear algebra, hence is independent of the structure of $M$ as a Dieudonn\'e module. The 
following characterization of crystalline discriminants within $\{1,t^2\}$ will be useful, see \cite[Corollary(3.5)]{ogus} for a proof:

\begin{fact}[Ogus]
\label{ogus}
Assume that $\phi$ is a non-degenerate symmetric pairing on a Dieudonn\'e 
module $M$ of rank $2n$. Assume also that there exists a $\phi$-isotropic 
$K(k)$-subspace $A\subset M\otimes\q$ of dimension $n$. Then 
$$\cds(M,\phi)=(-1)^nt^{2\dim_{K(k)}(A+FA/A)},$$ 
in particular, if $A$ is an isocrystal, then $\cds(M,\phi)=(-1)^n$.
\end{fact}

Recall that the Grassmannian of $n$-dimensional isotropic subspaces of the $K(k)$-vector space $M\otimes\q$ has two connected components. Two such 
spaces $A_1$ and $A_2$ lie in the same component if and only if the integer $\dim_{K(k)}(A_1+A_2/A_1)$ is even, we will say that $A_1$ and $A_2$ have the 
same parity if this is the case. Thus by the above fact $(-1)^n\cds(M,\phi)$ is the trivial element of $K(\f_p)^\times/(K(\f_p)^\times)^2$ if and only if 
the bijection $F$ does not change the parity of the maximal isotropic subspaces. If $p$ is odd and if $\phi$ is a perfect form on $M$, then one can 
deduce a further formulation: Pick a maximal isotropic $k$-subspace $\Abar\subset\Mbar=M/pM$ complementary to $V\Mbar$. Lift it to $W(k)$ to 
obtain a maximal isotropic $W(k)$ submodule $A$ of $M$ (the Grassmannian is smooth). Observe that $F\Abar=F\Mbar$ to conclude that $\pmod2$:

\begin{eqnarray*}
&&\dim_k(\Abar+F\Abar/\Abar)\\
&\equiv&\dim_k(\Abar+V\Mbar/\Abar)+\dim_k(V\Mbar+F\Mbar/V\Mbar)\\
&\equiv&\dim_k(\Mbar/V\Mbar+F\Mbar)\pmod2.
\end{eqnarray*}

The integer $\dim_k(M/VM+FM)$ is called the Oort invariant of $M$ and denoted $a(M)$. 
Thus, we have derived the consequence, implicitely stated in \cite[Section 5.3]{moonen}:

\begin{fact}[Moonen]
\label{odd}
Assume that $\phi$ is a perfect symmetric pairing on the Dieudonn\'e module 
$M$ of rank $2n$. If $p$ is odd, then $\cds(M,\phi)=(-1)^nt^{2a(M)}$.
\end{fact}

The Dieudonn\'e module $M$ is called superspecial if it satisfies $FM=VM$, i.e. if $\rk_{W(k)}(M)=2a(M)$. 
Superspecial Dieudonn\'e modules may conveniently be described in terms of their skeletons, these are the 
$W(\f_{p^2})$-submodules defined by $\tilde M=\{x\in M|Fx=Vx\}$. We write $\O_\b$ for the ring extension 
of $W(\f_{p^2})$, obtained by adjoining an indeterminate $\sigma$ subject to the relations $\sigma^2=p$ and 
$\sigma a={^Fa}\sigma$, it operates in a self-explanatory way on $\tilde M$. As remarked in ~\cite{li} 
the assignment $M\mapsto\tilde M$ sets up an equivalence of the category of superspecial 
Dieudonn\'e modules with the category of finitely generated torsion free $\O_\b$-modules. 
We also write $\b$ for $\O_\b\otimes\q$, it is the unique non-split quaternion algebra 
over $K(\f_p)$. Observe that $\O_\b$ is the maximal order of $\b$. We let $\gm_\b$ 
be the maximal ideal of $\O_\b$, one has $\O_\b/\gm_\b\cong\f_{p^2}$.\\

We need to put pairings into the picture as follows. If $\phi$ is a pairing on a superspecial 
Dieudonn\'e module $M$, then one considers a $\O_\b$-valued pairing on $\tilde M$ defined by:

\begin{equation}
\label{phi}
\Phi(x,y)=\phi(x,\sigma ty)-\phi(x,y)\sigma t
\end{equation}

This is $\O_\b$-sesquilinear, i.e. satisfies $\Phi(ux,vy)=u\Phi(x,y)v^\iota$, for $u,v\in\O_\b$. The involution $\iota$ is the standard one, mapping 
$a+b\sigma$ to ${^Fa}-b\sigma$. Conversely any $\O_\b$-sesquilinear form arises from a pairing on $M$ in the way described, $\phi$ is 
non-degenerate/perfect if and only if $\Phi$ is.\\
Unless otherwise said we assume from now on that $\phi$ is symmetric, in terms of $\Phi$ this means $\Phi(y,x)^\iota=-\Phi(x,y)$ for all 
$x,y\in\tilde M$. The $\O_\b$-module $\tilde M$ with form $\Phi$ is called hyperbolic if on a suitable $\O_\b$-basis $e_1,\dots,e_{n/2},f_1,\dots,f_{n/2}$ 
of $\tilde M$ one has
\begin{eqnarray}
\label{hyp}
\Phi(e_i,e_j)=\Phi(f_i,f_j)=0&,&\Phi(e_i,f_j)=w\delta_{i,j}
\end{eqnarray}
for some non-zero $w\in\O_\b$, uniquely determined only up to multiplication by $\O_\b^\times$. It turns out that $w=-\sigma^rt$ is a very convenient 
choice as the values of the corresponding form $\phi$ will then read:
$$\phi(e_i,f_j)=\begin{cases}0&r\equiv0\pmod2\\
p^{r-1/2}\delta_{i,j}&r\equiv1\pmod2\end{cases}$$
and
$$\phi(e_i,Ff_j)=\begin{cases}p^{r/2}\delta_{i,j}&r\equiv0\pmod2\\
0&r\equiv1\pmod2\end{cases}$$
and $\phi(e_i,e_j)=\phi(f_i,f_j)=0$. Equivalently, $(M,\phi)$ is hyperbolic if and only if the Dieudonn\'e module $M$ allows a decomposition into a 
direct sum of Dieudonn\'e modules $A$ and $B$ with $\phi(A,A)=\phi(B,B)=0$, and $M^t=F^{-r}M$, so that $\phi$ identifies the dual of $A$ with $F^{-r}B$.

\subsection{Results of Oort and Li}

\label{locpre}

A Dieudonn\'e module is called supersingular if it is isogenous to a superspecial one, or equivalently if all its Newton slopes are equal to 
$1/2$. This section is primarily concerned with supersingular Dieudonn\'e modules, so recall some of the techniques which are usually applied to them: 
If $M$ is supersingular it has a biggest superspecial sub-module $S_0(M)$ which one can construct as $S_0(M)=\tilde M\otimes_{W(\f_{p^2})}W(k)$. Dually 
there is $S^0(M)$, the smallest superspecial module containing $M$, see ~\cite[Chapter III.2]{manin} for proofs of this.\\
The following facts on the relation of the lattices $S_0(M)\subset M\subset S^0(M)$ are basic to the study of supersingular Dieudonn\'e modules. The first 
of them can be found in ~\cite[Corollary(1.7)]{li}, along with more information on the functors $S_0$ and $S^0$. For the other two facts we refer the reader 
to ~\cite[Lemma(1.5/1.6)]{li} (or ~\cite[Fact(5.8)]{lioort}) and ~\cite[1.10(i)]{li} (or ~\cite[Chapter(12.2)]{lioort}):\\

\begin{fact}[Li]
\label{length}
Let $M$ be a supersingular Dieudonn\'e module of rank $2g$ over $W(k)$. Then one has $F^{g-1}S^0(M)=\sum_{i+j=g-1}F^iV^jM$. It
follows that $F^{g-1}S^0(M)\subset S_0(M)$, in particular the length of the $W(k)$-module $S^0(M)/S_0(M)$ is bounded by $g(g-1)$
and equality is acquired if and only if $a(M)=1$.
\end{fact}

\begin{fact}[Li]
\label{generator}
Let $N$ be a superspecial Dieudonn\'e module of rank $2g$ over $W(k)$. Let $x$ be an element of $N$. Then one has $S^0(W(k)[F,V]x)=N$ if and only if the 
elements
$$F^{g-1}x,F^{g-2}Vx,\dots,FV^{g-2}x,V^{g-1}x$$
form a basis of the $k$-vector space $F^{g-1}N/F^gN$. Moreover, an element
with this property exists.
\end{fact}

\begin{fact}[Li]
\label{inequality}
Let $M$ be a supersingular Dieudonn\'e module of rank $2g$ over $W(k)$. For a 
non-negative integer $i$ let $$s_i=\dim_k(M\cap F^iS^0(M)/M\cap F^{i+1}S^0(M)).$$
Then one has $s_i\leq s_{i+1}$ and equality holds if and only if $s_i=g$.
\end{fact}

The work ~\cite{lioort} studies supersingular Dieudonn\'e modules which are equipped with a perfect anti-symmetric form $\psi$. Following their method we 
notice that we have to incorporate additional structure which by the Morita-equivalence of subsection ~\ref{Fermat} leads to Dieudonn\'e modules 
$M$ of rank $8$ equipped with a symmetric form $\phi$.\\
Analogous to ~\cite[Proposition(6.1)]{lioort} we need to analyze the restriction of $\phi$ to $N=S_0(M)$, or more generally, a classification of 
non-degenerate symmetric forms on superspecial Dieudonn\'e modules:

\begin{thm}
\label{diag}
Let $k$ be an algebraically closed field of characteristic $p\neq2$ and let $N$ be a superspecial Dieudonn\'e module of rank $2n$ over $W(k)$, which is 
equipped with a non-degenerate symmetric pairing $\phi$. Then $N$ contains Dieudonn\'e modules $N_i$ of rank $2$, with $\phi(N_i,N_j)=0$, for $i\neq j$, 
and $N=\bigoplus_{i=1}^nN_i$. Moreover, each $N_i$ has a $W(k)$-basis consisting of elements $x_i,Fx_i=Vx_i=y_i$ such that one of the two cases:
\begin{itemize}
\item[(i)]
$\phi(x_i,x_i)=\phi(y_i,y_i)=0$, and $\phi(x_i,y_i)=p^{n_i}$,
\item[(ii)]
$\phi(x_i,y_i)=0$, $\phi(x_i,x_i)=\epsilon_ip^{n_i}$, and 
$\phi(y_i,y_i)=\epsilon_i^\sigma p^{n_i+1}$,
\end{itemize}
holds for some integers $n_i$ and some elements 
$\epsilon_i\in W(\f_{p^2})^\times$ which are unique up to multiplication 
by elements in $(W(\f_{p^2})^\times)^2$. Moreover, the cristalline 
discriminant can be computed from this decomposition as
$$
\cds(N_i,\phi)=\begin{cases}-t^2&(N_i,\phi)\text{ of type (i)}\\
pt^2\epsilon_i\epsilon_i^\sigma&(N_i,\phi)\text{ is of type (ii)}
\end{cases},
$$
and $\cds(N,\phi)=\prod_{i=1}^n\cds(N_i,\phi)$.
\end{thm}
\begin{proof}
The skeleton construction descends $N$ to a $W(\f_{p^2})$-Dieudonn\'e module $\tilde N$ which at the same time is a $\O_\b$-module. As in ~\eqref{phi} we 
consider the $\O_\b$-valued sesquilinear form $\Phi$ and diagonalize it as follows: Let $x_0\in\tilde N$ be an element with $\Phi(x_0,x_0)$ of 
$\gm_\b$-adic valuation as small as possible, i.e. such that $\Phi(x,x)\in\gm_\b^r=\O_\b\Phi(x_0,x_0)$ for all $x\in\tilde N$. By the usual polarization process it follows that
$\Phi(x,y)-\Phi(x,y)^\iota\in\gm_\b^r$, and also $\Phi(x,y)+\Phi(x,y)^\iota\in\gm_\b^r$ by replacing $tx$ for $x$. Consequently $\Phi(\tilde N,\tilde N)\subset\gm_\b^r$. Therefore 
we obtain an orthogonal direct sum $\tilde N=(\tilde N\cap(\b x_0)^\perp)\oplus\O_\b x_0$, as any $x\in\tilde N$ has $\Phi(x,x_0)\Phi(x_0,x_0)^{-1}=\alpha\in\O_\b$ 
which allows to write $x$ as a sum of $\alpha x_0\in\tilde N$ and $x-\alpha x_0\in\tilde N\cap(\b x_0)^\perp$.\\
Having obtained a decomposition $\tilde N=\bigoplus_{i=1}^n\tilde N_i$ we search for basis elements $\tilde x_i\in\tilde N_i$ with $\Phi(\tilde x_i,\tilde x_i)$ manageable: 
In $\tilde N_i\otimes\q$ one can certainly find elements $\tilde x_i$ with $\Phi(\tilde x_i,\tilde x_i)\in W(\f_{p^2})^\times\cup FW(\f_{p^2})^\times$ 
for example by ~\cite[Chapter 10, Theorem(3.6.(i))]{scharlau}. Observe that the $\gm_\b$-adic valuation of $\Phi(\tilde x_i,\tilde x_i)$ must be 
congruent modulo $2$ to $r_i=\lg_{\O_\b}\tilde N_i^t/\tilde N_i$. Hence after adjusting the $\tilde x_i$'s by multiplying them by $F^{r_i/2}$, if $r_i$ is 
even, and by $F^{(r_i-1)/2}$, if $r_i$ is odd, one gets generators of the $\O_\b$-modules $\tilde N_i$ on which the sesquilinear form takes values in 
$F^{r_i}W(\f_{p^2})^\times$.\\
It is clear how to obtain the desired basis $x_1,\dots,x_n,y_1,\dots,y_n$ from these generators. If $r_i$ is even $N_i$ will be of type (i) with 
$n_i=r_i/2$, and if $r_i$ is odd then $N_i$ will be of the type (ii) with $n_i=(r_i-1)/2$.
\end{proof}

\begin{rem}
\label{plane}
Suppose $N$ is a superspecial Dieudonn\'e module of rank $2$ with a symmetric form $\phi$. Then one checks from the above classification that $(N,\phi)$ is 
isometric to $(N,-\phi)$. It follows that $N^{\oplus2}$, the orthogonal direct sum of two copies of $N$, is hyperbolic. One checks this by using the 
sesquilinear form ~\eqref{phi} as 
$\Phi((u_1+u_2,u_1-u_2),(v_1+v_2,v_1-v_2))=(u_1+u_2)w(v_1+v_2)^\iota-(u_1-u_2)w(v_1-v_2)^\iota=(2u_1w)v_2^\iota-u_2(2v_1w)^\iota$. 
(cf. ~\cite[Remark(6.1)]{lioort} for the analog in the anti-symmetric setting)
\end{rem}

For later use we note an immediate corollary:

\begin{cor}
\label{cdsdef}
Let $(N_1,\phi_1)$ and $(N_2,\phi_2)$ be supersingular Dieudonn\'e modules of rank two, equipped 
with symmetric pairings. There exists an isometry between them if and only if the following holds:
\begin{eqnarray*}
\lg_{W(k)}N_1^t/N_1&=&\lg_{W(k)}N_2^t/N_2\\
\cds(N_1,\phi_1)&=&\cds(N_2,\phi_2).
\end{eqnarray*}
Consequently for any non negative integer $n$, there is only one isometry class of rank 
two supersingular Dieudonn\'e modules with pairing $(N,\phi)$ where $\lg_{W(k)}N^t/N=2n$. 
There are two such classes of modules with pairing where $\lg_{W(k)}N^t/N=2n+1$.
\end{cor}

\subsection{Classification of symmetric Dieudonn\'e modules}
\label{Alocres}

This section is the core of the work, we give a classification 
of Dieudonn\'e modules with the additional structure of interest.

\begin{thm}
\label{eight}
Let $M$ be a supersingular Dieudonn\'e module over $W(k)$ with perfect symmetric pairing $\phi$. Assume that:
\begin{eqnarray*}
&&\rk_{W(k)}M=8\\
&&\cds(M,\phi)=1.
\end{eqnarray*}
Consider $S^0(M)=N$, the smallest superspecial Dieudonn\'e lattice in $M\otimes\q$, which contains $M$. Choose a decomposition 
$N=\bigoplus_{i=1}^4N_i$ with properties as granted by Theorem ~\ref{diag}, and with $S_0(M)=N^t=\bigoplus_{i=1}^4F^{r_i}N_i$ for integers 
$r_1\leq r_2\leq r_3\leq r_4$. Then $(r_1,r_2,r_3,r_4)$ is one of
\begin{itemize}
\item[(i)]
$(0,0,0,0)$
\item[(ii)]
$(1,1,1,1)$
\item[(iii)]
$(0,2,2,2)$
\item[(iv)]
$(2,2,2,2),$
\end{itemize}
moreover, there exists a superspecial Dieudonn\'e lattice $Q$, which contains $FM$ and satisfies
\begin{itemize}
\item[(a)]
$Q^t=Q$
\item[(b)]
$\dim_k(M/M\cap Q)=\dim_k(Q/M\cap Q)=1$.
\end{itemize}
If $M$ is of the form (iii) or (iv), then the superspecial Dieudonn\'e lattice $Q$, satisfying (a) and 
(b) is unique.
\end{thm}
\begin{proof}

For the proof we need two auxiliary lemmas:

\begin{lem}
\label{first}
Let the assumptions on $M$ be as in the above theorem, then there exist two different indices $i_1$ and $i_2$ such that $N_{i_1}$ and $N_{i_2}$ are isometric.
\end{lem}
\begin{proof}
If an even integer $r$ occurs twice amongst the various $r_i$'s one is done, and if an odd integer $r$ occurs three times one is done as well, use the 
pigeon hole principle and Corollary ~\ref{cdsdef}. The condition on the discriminant forces the number of indices $i$ with $r_i$ odd to be even. This 
means that one is left with checking the lemma for the $r_i$-quadruples $(0,1,2,3)$, $(0,2,3,3)$, $(0,1,1,2)$, and $(1,1,3,3)$.\\
The three quadruples with $r_1=0$ do not arise, because otherwise $M$ would be an orthogonal direct sum of $N_1$ and some supersingular Dieudonn\'e 
module $M'$ of rank $6$ and equipped with a perfect symmetric form $\phi'$. Applying Fact ~\ref{odd} to $M'$ would give that $M'$ has 
Oort invariant $1$ or $3$, as $\cds(M')=\cds(N_1)=-t^2$. Fact ~\ref{length} applied to $M'$ would further imply that the elementary divisors of
$S^0(M')/S_0(M')$ are either all $0$ or all equal to $2$. Hence the elementary divisors of $S^0(M)/S_0(M)$ would be $(0,0,0,0)$ or $(0,2,2,2)$.\\
It remains to do the $(r_1,r_2,r_3,r_4)=(1,1,3,3)$-case. Assume that no two of the $N_i's$ were isometric. This would lead to a basis $x_i,Fx_i=Vx_i=y_i$ with
\begin{eqnarray*}
&&\phi(x_1,x_1)=p^{-1}\mbox{, }\phi(y_1,y_1)=1\\
&&\phi(x_2,x_2)=\epsilon p^{-1}\mbox{, }\phi(y_2,y_2)=\epsilon^\sigma\\
&&\phi(x_3,x_3)=p^{-2}\mbox{, }\phi(y_3,y_3)=p^{-1}\\
&&\phi(x_4,x_4)=\epsilon p^{-2}\mbox{, }\phi(y_4,y_4)=\epsilon^\sigma p^{-1},\\
&&\mbox{ other products}=0,
\end{eqnarray*}
and with $\epsilon$ some non-square in $W(\f_{p^2})^\times$. The module $M$ has to contain an element of the form
$\alpha_1x_1+\alpha_2x_2+\alpha_3x_3+\alpha_4x_4+\beta_3y_3+\beta_4y_4$ such that $\beta_i,\alpha_i\in W(k)$ but not both of $\alpha_3$ and $\alpha_4$ in $pW(k)$. As
$$
\phi(x,x)=p^{-1}
(\alpha_1^2+\epsilon\alpha_2^2+\beta_3^2+\epsilon^\sigma\beta_4^2)+
p^{-2}(\alpha_3^2+\epsilon\alpha_4^2)
$$
one has $\alpha_3^2+\epsilon\alpha_4^2\equiv0\pmod p$, but as
$$
\phi(x,F^2x)
=\alpha_1\alpha_1^{\sigma^2}+\epsilon\alpha_2\alpha_2^{\sigma^2}
+\beta_3\beta_3^{\sigma^2}+\epsilon^\sigma\beta_4\beta_4^{\sigma^2}
+p^{-1}(\alpha_3\alpha_3^{\sigma^2}
+\epsilon\alpha_4\alpha_4^{\sigma^2})
$$
one has $\alpha_3^{p^2+1}+\epsilon\alpha_4^{p^2+1}\equiv0\pmod p$ as well. As $\epsilon$ is a 
non-square in $W(\f_{p^2})^\times$, one has $\epsilon^\frac{p^2-1}{2}\equiv-1\pmod p$, so that we 
derive the contradiction 
$$\alpha_3^{p^2+1}
\equiv(-\epsilon\alpha_4^2)^\frac{p^2+1}{2}
\equiv\epsilon\alpha_4^{p^2+1}\pmod p.$$
\end{proof}

\begin{lem}
\label{second}
With the same notation as in the theorem $r_i\leq2$ for all indices $i$.
\end{lem}
\begin{proof}
Observe that the lemma would be immediate if one of the $r_i$ was zero. So we can assume $0<r_i$ for all indices $i$. 
Pick two indices $i\neq j$ with $r_i=r_j=r$ and $\cds(N_i)=\cds(N_j)$, according to the previous lemma such indices 
will exist. Say $(i,j)=(1,2)$ after relabeling, and write according to Remark ~\ref{plane} $N_1\oplus N_2=A\oplus B$, 
with $\phi(A,A)=\phi(B,B)=0$ and $A\times F^rB\rightarrow W(k)$ a perfect pairing. Consider along the lines of 
~\cite[Proposition(6.3)]{lioort} a $W(k)$-module $M'$ which is the image of $(B\oplus N_3\oplus N_4)\cap M$ under the 
projection map $B\oplus N_3\oplus N_4\rightarrow N_3\oplus N_4$. $M'$ inherits a perfect form and is indeed canonically 
isomorphic to the sub-quotient $(B^\perp\cap M)/(B\cap M)$ of $M$. One has $\cds(M')=1$ because $M'$ is isogenous to 
$N_3\oplus N_4$. By Fact ~\ref{odd} it follows that $M'$ is superspecial. Furthermore the proof of 
~\cite[Proposition(6.3)]{lioort} shows that $FN_3\oplus FN_4\subset M'\subset N_3\oplus N_4$. For convenience of the
reader we reproduce the argument in loc.cit.: Pick an element in $M$ of the form $x=e+f+n_3+n_4$ with $e\in\tilde A$, 
$f\in B$, $n_3\in N_3$, $n_4\in N_4$ and
$$
S^0(M)=S^0(W(k)[F,V]x),
$$
it exists due to Fact ~\ref{generator}. The elements $F^3x$, $F^2Vx$, $FV^2x$, $V^3x$ will then form a basis of the $k$-vector space $F^3N/F^4N$ so 
that $F^3x-F^2Vx$, $F^2Vx-FV^2x$, $FV^2x-V^3x$ is a basis of $F^3(B\oplus N_3\oplus N_4)/F^4(B\oplus N_3\oplus N_4)$. It follows that
$$
S^0(W(k)[F,V](F-V)x)=F(B\oplus N_3\oplus N_4),
$$
but $(F-V)x\in M\cap(B\oplus N_3\oplus N_4)$ which projects surjectively onto $M'$. As $S^0$ is a functor in supersingular Dieudonn\'e modules 
$FN_3\oplus FN_4$ will be contained in $S^0(M')=M'$, and consequently
$$
FN_3\oplus FN_4\subset M^\prime
=M^{\prime t}\subset F^{-1}N_3^t\oplus F^{-1}N_4^t
=F^{r_3-1}N_3\oplus F^{r_4-1}N_4
$$
i.e. $r_3,r_4\leq2$. However, $r_3\equiv r_4\pmod2$, as $\cds(N_3)=\cds(N_4)$. Therefore $r_3=r_4$, as $r_3,r_4\in\{1,2\}$. Now, note that this does indeed 
imply that $N_3$ is isometric to $N_4$.\\
In order to find that $r_1,r_2\leq 2$ also, we redo the whole argument, with the roles of $N_1$ and $N_2$ being replaced by $N_3$ and $N_4$.
\end{proof}

{\em Return to proof of theorem\/\\}

We move on to investigate the set of possible quadruples $(r_1,r_2,r_3,r_4)$. If one of the numbers in that sequence is $0$, then Fact ~\ref{odd} shows 
that we must have either $(0,0,0,0)$ or $(0,2,2,2)$. For the remaining cases $(2,2,2,2)$, $(1,1,1,1)$ and $(1,1,2,2)$ are conceivable. We show that
$(1,1,2,2)$ can not arise: Assume we had a Dieudonn\'e module $M$ with $(r_1,r_2,r_3,r_4)=(1,1,2,2)$. It would follow that one had 
$\cds(N_3)=\cds(N_4)$ by Corollary ~\ref{cdsdef}, and so would $\cds(N_1)=\cds(N_2)$. By applying Remark ~\ref{plane} to both 
$N_1\oplus N_2$ and $N_3\oplus N_4$ one obtains a basis of $N$ consisting of say $e_1$, $e_2$, $f_1$, $f_2$, $Fe_1=Ve_1$, $Fe_2=Ve_2$, $Ff_1=Vf_1$, 
$Ff_2=Vf_2$ and with the only non-zero products being given by
\begin{eqnarray*}
&&\phi(Fe_1,Ff_1)=1\\
&&\phi(e_1,f_1)=\phi(e_2,Ff_2)=\phi(f_2,Fe_2)=p^{-1}.
\end{eqnarray*}
As $F^{-1}N^t$ is superspecial one has $M\not\subset F^{-1}N^t$, so that $M$ contains an element of the form
$x=\alpha_1e_1+\beta_1f_1+\alpha_2e_2+\beta_2f_2+\alpha_3Fe_2+\beta_3Ff_2$, with all $\alpha_1,\dots,\beta_3\in W(k)$ and at least one of $\alpha_2$ and 
$\beta_2$ a unit. From $Fx\in\alpha_2^\sigma Fe_2+\beta_2^\sigma Ff_2+N^t$ and $\phi(M,M)\subset W(k)$ one infers 
$\phi(x,Fx)\in p^{-1}(\alpha_2^\sigma\beta_2+\beta_2^\sigma\alpha_2)+W(k)$, which means that
$\alpha_2^\sigma\beta_2+\beta_2^\sigma\alpha_2\equiv0\pmod p$. As we may alter the elements $\alpha_1,\dots,\beta_3$ by any element in $pW(k)$ we can 
actually assume that $\alpha_2^\sigma\beta_2+\beta_2^\sigma\alpha_2=0$, but 
then the Dieudonn\'e module
$$
W(k)Fx+N^t=W(k)(\alpha_2^\sigma Fe_2+\beta_2^\sigma Ff_2)+N^t
$$
is superspecial contradicting $S_0(M)=N^t$.\\
Having done the first assertion of the theorem we now focus on the existence of $Q$. If $M$ is of the form (i), then use Remark ~\ref{plane} to write 
$N_1\oplus N_2$ as direct sum of two isotropic Dieudonn\'e modules $A$ and $B$, between which there is the duality that is induced from the pairing on 
$N$. Then one finds that $Q=F^{-1}A\oplus FB\oplus N_3\oplus N_4$ is a superspecial Dieudonn\'e lattice that does the job. Similarly for the 
(ii)-case: Write $N=A_1\oplus A_2\oplus B_1\oplus B_2$ with isotropic $A_i$ and $B_i$, this time equipped with a canonical isomorphism $A_i^t\cong FB_i$.
The superspecial lattices
\begin{eqnarray*}
&&FA_1\oplus FA_2\oplus B_1\oplus B_2\\
&&A_1\oplus FA_2\oplus FB_1\oplus B_2
\end{eqnarray*}
both satisfy $Q^t=Q$, and one of them satisfies property (b) as well.\\
In the (iii)-case property (a) forces to look at $Q=N_1\oplus\bigoplus_{i=2}^4FN_i$, whereas $Q=FN$ in the (iv)-case. We have 
to show that this module does indeed satisfy (b), to this end observe that the numbers $\dim_kM/M\cap Q$ and $\dim_kM\cap Q/M\cap FQ$ are nonzero and sum 
up to $4$, it thus suffices to see that the first of them is strictly smaller than the second. In the (iv)-case this is the content of Fact 
~\ref{inequality}. In the (iii)-case apply Fact ~\ref{inequality} to the orthogonal complement of $N_1$ in $M$, which is a Dieudonn\'e module of rank 
$6$ with perfect symmetric form.
\end{proof}

\subsection{Moduli of symmetric Dieudonn\'e modules}
\label{Blocres}
We consider the graded $\f_p$-algebra $R:=\f_p[A_1,A_2,B_1,B_2]/(\sum_{i=1}^2A_iB_i^p+B_iA_i^p)$, and its associated 
projective variety $\Xbar_1:=\Prj R$, which is smooth of relative dimension $2$. Let $\Ybar_1$ denote the affine chart 
determined by $A_1\neq0$, it is the spectrum of $R_{(A_1)}\cong\f_p[a_2,b_1,b_2]/(b_1+b_1^p+a_2b_2^p+b_2a_2^p)$, where 
$a_2:=\frac{A_2}{A_1}$, $b_1:=\frac{B_1}{A_1}$, and $b_2:=\frac{B_2}{A_1}$. Let 
$\alpha_2,\beta_1,\beta_2\in W(R_{(A_1)})$ be lifts of $a_2,b_1,b_2$ with 
$\beta_1+{^F\beta_1}+\alpha_2{^F\beta_2}+\beta_2{^F\alpha_2}=0$. Let $T_{(A_1)}$ be the $W(R_{(A_1)})$-module 
$\bigoplus_{i=1}^4W(R_{(A_1)})t_i$, $L_{(A_1)}$ be the $W(R_{(A_1)})$-module $\bigoplus_{i=1}^4W(R_{(A_1)})l_i$, and 
$M_{(A_1)}$ be $L_{(A_1)}\oplus T_{(A_1)}$. Putting:
\begin{eqnarray*}
&&F(t_1)=l_1\\
&&F(t_2)=l_2+(\beta_2-{^{F^2}\beta_2})t_1\\
&&F(t_3)=l_3+({^{F^2}\alpha_2}-\alpha_2)t_2+({^{F^2}\beta_2}-\beta_2)t_4\\
&&F(t_4)=l_4+(\alpha_2-{^{F^2}\alpha_2})t_1\\
&&V^{-1}(l_1)=t_1\\
&&V^{-1}(l_2)=t_2\\
&&V^{-1}(l_3)=t_3\\
&&V^{-1}(l_4)=t_4
\end{eqnarray*}
and using the formula $V^{-1}(^V\alpha x)=\alpha F(x)$ defines the structure of a display (\cite{zink2}) on 
$M_{(A_1)}$, which moreover has the normal decomposition $L_{(A_1)}\oplus T_{(A_1)}$. One checks that a pairing is 
given on $M_{(A_1)}$ by $\phi(l_i,l_j)=\phi(t_i,t_j)=0$, $\phi(l_i,t_j)=\delta_{|i-j|,2}$. Let also $N=L_N\oplus T_N$ 
be the display obtained from the formulas $F(t_i)=l_i$, $V^{-1}(l_i)=t_i$ and with pairing defined analogously. Putting:
\begin{eqnarray*}
&&\epsilon(t_1)=pt_3\\
&&\epsilon(t_2)=l_2-{^F\beta_2l_3}\\
&&\epsilon(t_3)=t_1+{^F\alpha_2t_2}+{^F\beta_1t_3}+{^F\beta_2t_4}\\
&&\epsilon(t_4)=-{^F\alpha_2l_3}+l_4\\
&&\epsilon(l_1)=pl_3\\
&&\epsilon(l_2)=pt_2-p\beta_2 t_3\\
&&\epsilon(l_3)=l_1+\alpha_2l_2+\beta_1l_3+\beta_2l_4\\
&&\epsilon(l_4)=-p\alpha_2t_3+pt_4
\end{eqnarray*}
defines an embedding of displays $\epsilon_{(A_1)}:M_{(A_1)}\hookrightarrow N\times_{\f_p}\Ybar_1$, 
satisfying $p\phi(x,y)=\phi(\epsilon(x),\epsilon(y))$. Neither $M_{(A_1)}$ nor $\epsilon_{(A_1)}$ 
depend on the choice of the lifts $\alpha_2,\beta_1,\beta_2$, which can be checked upon passage to 
the perfection $R_{(A_1)}^{perf}$ (here notice that $R_{(A_1)}\rightarrow R_{(A_1)}^{perf}$ is flat, 
because $R_{(A_1)}$ is regular). Moreover, the natural action of the Kleinian group on $\Xbar_1$ 
gives rise to analogous subdisplays of the constant display $N$ regarded over each of the translates 
$\{A_2\neq0\}$, $\{B_1,\neq0\}$, and $\{B_2\neq0\}$, which in turn gives rise to an inclusion 
$$\epsilon:M\hookrightarrow N\times_{\f_p}\Xbar_1,$$ 
of sheaves of displays with respect to the Zariski topology of $\Xbar_1$. This is because the closed 
points can be used to check the cocycle condition. However, notice that there does not exist a global 
normal decomposition for $M$.

\subsection{Miscellaneous}

\label{fibre}
The study of families of Dieudonn\'e modules with our additional structure within a given isogeny class is meaningful 
not just for the supersingular one. Recall that every isogeny class of Dieudonn\'e modules can be written as a direct 
sum of certain simple ones. These are parameterized by pairs of coprime non-negative integers $a$ and $b$ and denoted 
by $G_{a,b}$, see ~\cite{manin} for details. The isogeny class $G_{a,b}$ contains usually more than one Dieudonn\'e 
module except if $a$ or $b$ is equal to $1$, in which case we are allowed to speak of ``the'' Dieudonn\'e module of 
type $G_{a,b}$. We have the following result:\\

\begin{cor}
\label{blow}
Let $M$ be a non-supersingular Dieudonn\'e module over $W(k)$ that is equipped with a perfect symmetric pairing $\phi$. Assume that:
\begin{eqnarray*}
&&\rk_{W(k)}M=8\\
&&\cds(M,\phi)=1.
\end{eqnarray*}
Then $M$ is an orthogonal direct sum $\bigoplus_iM_i$ where for each of the $(M_i,\phi)$ one of the following alternatives hold:
\begin{itemize}
\item[(i.n)]
$(M_i,\phi)$ can be written as $A\oplus B$ with mutually dual isotropic Dieudonn\'e modules $A$ and $B$, which lie in the isogeny classes $G_{1,n}$ 
and $G_{n,1}$ for some $n\in\{0,1,2,3\}$.
\item[(ii)]
$(M_i,\phi)$ is supersingular of rank $2$ and the perfect pairing thereon is the one described by part $(i)$ of theorem ~\ref{diag}.
\item[(iii)]
$(M_i,\phi)$ is supersingular of rank $4$, and the pairing is such that $S_0(M_i)$ decomposes into the two Dieudonn\'e modules with pairings 
described by part $(ii)$ of theorem ~\ref{diag}.
\end{itemize}
Moreover, the only combinations which occur are:

\begin{itemize}
\item
$4\times\operatorname{(i.0)}$
\item
$2\times\operatorname{(i.0)}\oplus\operatorname{(i.1)}$
\item
$\operatorname{(i.0)}\oplus\operatorname{(ii)}\oplus\operatorname{(iii)}$
\item
$\operatorname{(i.0)}\oplus\operatorname{(i.2)}$
\item
$\operatorname{(i.3)}$
\end{itemize}
\end{cor}

\begin{proof}
We consider the canonical decomposition of $M=M_0\oplus M'\oplus M_1$ into the \'etale-local, local-local, 
local-\'etale parts. The assertion of the corollary has solely something to do with $M'$ which is of some even 
rank equal to $8-2f$ and has $\cds(M')=(-1)^f$, here $f$ is the $p$-rank of $M$. As $M'$ is also self-dual it 
can have only one of the following isogeny types:
\begin{enumerate}
\item
$3\times G_{1,1}$
\item
$2\times G_{1,1}$
\item
$G_{1,2}\oplus G_{2,1}$
\item
$G_{1,3}\oplus G_{3,1}$
\end{enumerate}
If $M'$ has the above isogeny types 3., or 4. we deduce from ~\cite[Paragraph(16), Satz(3)]{kraft} and $a(M')=2$ that $M'$ is a direct 
sum of two Dieudonn\'e modules $A$ and $B$, each with Oort invariant equal to one. The assertion on the pairing is then immediate as 
neither $A$ nor $B$ is selfdual.\\
If $M'$ has isogeny type $2G_{1,1}$ it must be superspecial. Then use theorem ~\ref{diag} in conjunction with remark ~\ref{plane} to check that 
$M'$ has the shape $A\oplus B$ with isotropic $A$ and $B$.\\
In the case in which the isogeny type of $M'$ is $3G_{1,1}$, we have to work a bit harder: First consider a diagonalization of 
$S^0(M')=N=\bigoplus_{i=1}^3N_i$ with $S_0(M')=N^t=F^{r_i}N_i$. An analysis 
as in the proof of lemma ~\ref{first} yields that $(r_1,r_2,r_3)=(0,1,1)$, therefore the orthogonal direct summand $(N_1,\phi)$ has a complement with 
perfect form, say $M''$, its Oort invariant is $1$. Therefore $\cds(M'')=t^2$. As $r_2=r_3=1$ this implies that $\cds(N_2,\phi)$, and $\cds(N_3,\phi)$, are 
the two numbers $p$, and $pt^2$, which is what we wanted.
\end{proof}

\section{The Shimura variety $S_{K^p}$}
\label{datum}

\subsection{Further Notation}
\label{yoga}
Before we proceed we want to introduce the input data for our $PEL$-moduli problem: Fix once and for all a quaternion algebra $B$ over $\q$ and write $R$ 
for the set of places at which $B$ is non-split. Assume that $\infty\in R$, i.e. that $B_\r$ is definite. Let $p$ be a prime which is not in $\{2\}\cup R$ 
and choose a maximal $\z_{(p)}$-order $\O_B\subset B$, together with an isomorphism $\kappa_p:\z_p\otimes\O_B\cong\Mat_2(\z_p)$. The standard 
involution $b\mapsto b^\iota=\tr(b)-b$ preserves $\O_B$ and is positive.\\
Let $V$ be a left $B$-module of rank $4$ with non-degenerate alternating pairing satisfying $(bv,w)=(v,b^\iota w)$. For simplicity we require that the 
skew-Hermitian $B$-module $V$ is hyperbolic in the following sense: We want it to have a $B$-basis $e_1,e_2,f_1,f_2$ such that
$(\sum_{i=1}^2a_ie_i+b_if_i,\sum_{i=1}^2a_i^\prime e_i+b_i^\prime f_i)=\tr_{B/\q}(\sum_{i=1}^2a_ib_i^{\prime\iota}-b_ia_i^{\prime\iota})$ for all 
$a_i,b_i,a_i^\prime,b_i^\prime\in B$. Set further $\Lambda_0=\bigoplus_{i=1}^2\O_Be_i\oplus\O_Bf_i$, it is a self-dual $\O_B$-invariant $\z_{(p)}$-lattice in $V$.\\
Let $G/\q$ be the reductive group of all $B$-linear symplectic similitudes of $V$. This group is a form of $GO(8)$. 
Write $K_p\subset G(\q_p)$ for the hyperspecial subgroup consisting of group elements that preserve $\Lambda_0$ and let $K^p\subset G(\adp)$ be an 
arbitrary compact open subgroup.\\
Finally we specify a particular $*$-homomorphism $h_0:\c\rightarrow\End_B(V_\r)$ by the rule 
$h_0(i)(\sum_{i=1}^2a_ie_i+b_if_i)=\sum_{i=1}^2b_ie_i-a_if_i$, and $\r$-linear extension. The reflex field of $(G,h_0)$ is equal to $\q$.\\
Now, for every connected scheme $S_{K^p}/\z_{(p)}$ with a geometric base point $s$ we consider the set of $\z_{(p)}$-isogeny classes of quadruples 
$(A,\lambda,\i,\ebar)$ with:

\begin{itemize}
\item[(M1)]
$A$ is a $8$-dimensional abelian scheme over $S$ up to prime-to-$p$ isogeny
\item[(M2)]
$\lambda:A\rightarrow A^t$ is a $\z_{(p)}^\times$-class of prime-to-$p$ 
polarizations of $A$
\item[(M3)]
$\i:\O_B\rightarrow\End(A)\otimes\z_{(p)}$ is a homomorphism satisfying 
$\i(b^\iota)=\i(b)^*$, here $*$ is the Rosati involution
associated to $\lambda$
\item[(M4)]
$\ebar$ is a $\pi_1(S,s)$-invariant $K^p$-orbit of $\O_B$-linear isomorphisms 
$\eta:V\otimes\adp\cong H_1(A_s,\adp)$ which are compatible with the 
alternating form up to scalars.
\end{itemize}

By geometric invariant theory this functor is representable by a quasi-projective $\z_{(p)}$-scheme $S_{K^p}$. 
Moreover, the deformation theory of Grothendieck-Messing shows that $S$ is smooth of relative dimension $6$ over 
$\z_{(p)}$, cf. \cite[Chapter 5]{kottwitz}. See also \cite[Chapter 8]{kottwitz} for the complex uniformizations 
of $S_{K^p}(\c)$.\\

Finally, let us write $S_{K^p}^{si}$ (resp. $S_{K^p}^{sp}$) for the subsets 
$S_{K^p}\times\f_p^{ac}$ whose sets of geometric points consist of those quadruples $(A,\lambda,\i,\ebar)$ 
where $\d(A[p^\infty])$ is supersingular (resp. superspecial), here $\d(G)$ denotes the (covariant) 
Dieudonn\'e module of a $p$-divisible group $G$ over a perfect field. Notice that we always have 
$\cds(\d(G),\phi)=1$, by \cite{b}.

\subsection{Morita equivalence}
\label{Fermat}
Let us write $G^*$ for the Serre-dual of a $p$-divisible group $G=\bigcup_lG[p^l]$ over some base scheme 
$S$. We will say that $G$ is polarized (resp. anti-polarized) if it is endowed with an isomorphism $\phi$ 
to its dual which satisfies $\phi=-\phi^*$ (resp. $\phi=\phi^*$). In particular, consider the anti-polarized 
$p$-divisible groups $G_1:=\BT(M)$ and $G_0:=\BT(N)$, where $M$ and $N$ are as in section \ref{Blocres}. 
The emdedding $\epsilon:M\hookrightarrow N\times_{\f_p}\Xbar_1$ gives rise to a canonical isogeny 
$\epsilon:G_1\rightarrow G_0\times_{\f_p}\Xbar_1$ satisfying $\epsilon^*\circ\epsilon=p\id_{G_1}$ 
and $\epsilon\circ\epsilon^*=p\id_{G_0\times_{\f_p}\Xbar_1}$, notice also that 
$\ker(\epsilon)\subset G_1[p]$ and $\ker(\epsilon^*)\subset G_0[p]\times_{\f_p}\Xbar_1$ 
are finite, flat, maximal isotropic subgroup schemes of order $p^4$.\\
If an isomorphism $\z_p\otimes\O_B\stackrel{\kappa_p}{\rightarrow}\Mat_2(\z_p)$ 
is fixed once and for all, one obtains a Morita-equivalence 
$$(G,\phi)\mapsto(G^{\oplus2},\left(\begin{matrix}0&\phi\\-\phi&0\end{matrix}\right))$$ 
from the category of anti-polarized $p$-divisible groups to the category of polarized $p$-divisible groups with 
Rosati-invariant $\O_B$-action. In this manner one obtains an anti-polarized $p$-divisible group $(G,\phi)$ from 
every $S$-valued point on $S_{K^p}$, say represented by $(A,\lambda,\i,\ebar)$, by the requirement
$$(A[p^\infty],\psi_\lambda)\cong(G^{\oplus2},\left(\begin{matrix}0&\phi\\-\phi&0\end{matrix}\right)),$$ 
where $\psi_\lambda:A[p^\infty]\rightarrow A[p^\infty]^*$ is the $p$-adic Weil-pairing, which is induced from 
the polarization $\lambda:A\rightarrow A^t$. If $S$ is the spectrum of a perfect field of characteric $p$, we 
always have $\cds(\d(G),\phi)=1$, by \cite{b}. We next want to define a family of morphisms 
\begin{equation}
\label{para}
c_{x,\eta_p}:\Xbar_1\times\f_p^{ac}\rightarrow S_{K^p}\times\f_p^{ac}
\end{equation} 
which are indexed by superspecial $\f_p^{ac}$-points $x=(A,\lambda,\i,\ebar)$, equipped 
with the following additional datum: By a frame for $x$ we mean an isomorphism 
$\eta_p:G_0\times_{\f_p}\f_p^{ac}\rightarrow G$, where $(G,\phi)$ corresponds to $x\in S_{K^p}(\f_p^{ac})$ by 
the above Morita-equivalence while $(G_0,\phi_0)$ is the previously exhibited anti-polarized $p$-divisible 
group. Let us consider the abelian variety which is defined by the exact sequence:
$$0\rightarrow\eta_p(\ker(\epsilon^*))^{\oplus2}\rightarrow A\stackrel{e^t}{\rightarrow}A_1\rightarrow0,$$
the isotropicity and the $\O_B$-invariance of $\eta_p(\ker(\epsilon^*))^{\oplus2}$ give rise to a canonical 
$\z_{(p)}^\times$-class of prime-to-$p$ polarizations $\lambda_1:A_1\rightarrow A_1^t$, together with a 
Rosati-invariant operation $\i_1:\O_B\rightarrow\End(A_1)\otimes\z_{(p)}$ and level structure $\ebar_1$, each gotten 
by transport of structure. Finally one sees that the quadruple $x_1=(A_1,\lambda_1,\i_1,\ebar_1)$ thus obtained constitutes 
a $\Xbar_1\times\f_p^{ac}$-valued point, whose classifying morphism we define to be \eqref{para}. It is easy to see that 
the image of $c_{x,\eta_p}$ is a closed subset, whose geometric points consist of exactly those quadruples 
$(A_1,\lambda_1,\i_1,\ebar_1)$ which allow an $\O_B$-linear isogeny $e:A_1\rightarrow A$, wich is compatible with the level 
structure and satisfies $p\lambda_1=e^t\circ\lambda\circ e$.

\begin{rem}
\label{rich}
Fix $(A,\lambda,\i,\ebar)=x\in S_{K^p}^{sp}(\f_p^{ac})$. Notice, that we have just shown, that the Zariski-closed 
subset $c_{x,\eta_p}(\Xbar_1\times\f_p^{ac}):=S_{x,K^p}^{sp}$ does not dependent on the choice of 
frame.
\end{rem}

\subsection{Description of $S_{K^p}^{si}$}
Now, we would like to investigate whether or not $c_{x,\eta_p}$ is a closed immersion, 
the next lemma is a step towards this direction:

\begin{lem}
\label{immersed}
Let $x$ and $\eta_p$ be as above, then $c_{x,\eta_p}$ induces an injection on the tangentspaces to each 
geometric point $u\in\Xbar_1(k)$, where $k$ is an arbitrary algebraically closed field of characteristic $p$.
\end{lem}
\begin{proof}
Recall that every $k$-display $P$ of dimension $d$ and codimension $c$ allows structural equations:
\begin{eqnarray*}
&&F(t_j)=\sum_{i=1}^du_{i,j}t_i+\sum_{i=1}^cu_{i+d,j}l_i\\
&&V^{-1}(l_j)=\sum_{i=1}^du_{i,j+d}t_i+\sum_{i=1}^cu_{i+d,j+d}l_i
\end{eqnarray*} 
for some display-matrix 
$$\left(\begin{matrix}u_{1,c+d}&\dots&u_{1,c+d}\\
\vdots&\ddots&\vdots\\u_{c+d,1}&\dots&u_{c+d,c+d}\end{matrix}\right)=U\in\GL(c+d,W(k)),$$
where $t_1,\dots,t_d,l_1,\dots,l_c\in P$, and $t_1+Q,\dots,t_d+Q\in P/Q$ are bases. Let $L$ 
and $T$ be the $W(k)$-submodules of $P$ that are generated by $l_1,\dots,l_c$ and $t_1,\dots,t_d$, and
write $J:=\Hom_{W(k)}(L,T)$. Due to the technique of Norman-Oort the isomorphism classes of infinitesimal 
deformations of $P$ over the ring of dual numbers $k_D:=k[s]/(s^2)$ are parameterized by the elements in 
$J\otimes_{W(k)}k=\Hom_k(Q/pP,P/Q)$, in fact each deformation may be described explicitly as follows: Pick 
a tangent direction $N\in J\otimes_{W(k)}k$, say with $d\times c$-matrix representation
$$\left(\begin{matrix}n_{1,1}&\dots&n_{1,c}\\
\vdots&\ddots&\vdots\\n_{d,1}&\dots&n_{d,c}\end{matrix}\right)$$ (with respect to the two bases above). Write 
$W(sk_D)$ for the kernel of the natural map from $W(k_D)$ to $W(k)$, and choose elements $\tilde n_{i,j}\in W(sk_D)$ 
whose $0$-th Witt coordinate is equal to the dual number $sn_{i,j}$. Then
$$\tilde U:=
\left(\begin{matrix}1&\dots&0&\tilde n_{1,1}&\dots&\tilde n_{1,c}\\
\vdots&\ddots&\vdots&\vdots&\ddots&\vdots\\
0&\dots&1&\tilde n_{d,1}&\dots&\tilde n_{d,c}\\
0&\dots&0&1&\dots&0\\
\vdots&\ddots&\vdots&\vdots&\ddots&\vdots\\
0&\dots&0&0&\dots&1\end{matrix}\right)U\in\GL(c+d,W(k_D))$$
displays an infinitesimal deformation of $P$, that corresponds to the tangent direction $N$, in particular it 
is the trivial deformation if and only of $N=0$.\\
Now let $(X_1:X_2:Y_1:Y_2)$ be the homogeneous coordinates of $u\in\Xbar_1(k)$, and fix one of its non-zero 
tangent directions $u'\in\Xbar_1(k_D)$. To finish the proof of the lemma we only have to show that the associated 
$k_D$-display $M_{u'}$ is a non-trivial infinitesimal deformation (of $M_u$, i.e. the special fiber of $M_{u'}$). 
Of course we can assume $(X_1:X_2:Y_1:Y_2)=(1:x_2:y_1:y_2)$ from the start, so let 
$(1:x_2+sa:y_1-s(ay_2^p-bx_2^p):y_2+sb)$ be the homogeneous coordinates of $u'$, where 
$(a,b)\in k^2-\{(0,0)\}$. Now recall from section \ref{Blocres} that the restriction of $M$ to the affine chart 
$\Spc\f_p[a_2,b_1,b_2]/(b_1+b_1^p+a_2b_2^p+b_2a_2^p)\subset\Xbar_1$ has already a normal decomposition and 
is explicitly displayed in an extremely convenient way, namely by means of the matrix 
$U=\left(\begin{matrix}H&E\\E&0\end{matrix}\right)$, where $E$ denotes the identity matrix, and where the 
(so-called `Hasse-Witt') matrix $H$ is given by: 
$$\left(\begin{matrix}0&\beta_2-{^{F^2}\beta_2}&0&\alpha_2-{^{F^2}\alpha_2}\\
0&0&{^{F^2}\alpha_2}-\alpha_2&0\\
0&0&0&0\\
0&0&{^{F^2}\beta_2}-\beta_2&0\end{matrix}\right),$$
for certain $\alpha_2,\beta_1,\beta_2\in W(\f_p[a_2,b_1,b_2]/(b_1+b_1^p+a_2b_2^p+b_2a_2^p))$. Now consider the 
$sk_D$-valued Witt-vectors $\alpha:=u'(\alpha_2)-u(\alpha_2)$ and $\beta:=u'(\beta_2)-u(\beta_2)$, in fact it 
is easy to see that $u'(\beta_1)-u(\beta_1)=-(\alpha u(\beta_2)^\sigma+\beta u(\alpha_2)^\sigma)$, because 
$\alpha$ and $\beta$ are killed by $F$. Moreover, the $0$th Witt-coordinates of $\alpha$ and $\beta$ are just 
$sa$ and $sb$. It follows immediately that $u'(U)=\left(\begin{matrix}E&\tilde N\\0&E\end{matrix}\right)u(U)$, 
with $\tilde N$ being the deformation matrix:
$$\left(\begin{matrix}0&\beta&0&\alpha\\
0&0&-\alpha&0\\
0&0&0&0\\
0&0&-\beta&0\end{matrix}\right),$$
whose matrix of $0$th Witt-components is clearly nonvanishing.
\end{proof}

As a consequence of theorem \ref{eight} we have: 
$$S_{K^p}^{si}=\bigcup_{x\in S_{K^p}^{sp}}S_{x,K^p}^{si},$$ 
and $S_{K^p}^{sp}$ is a finite set of closed points. It follows from this (or from  
Grothendieck's specialization theorem \cite[p.149]{grothendieck}), that $S_{K^p}^{si}$ is Zariski closed. 
Our aim is to describe $S_{K^p}^{si}$ together with its induced reduced
subscheme structure. Let us fix $x\in S_{K^p}^{sp}$, which classifies
some quadruple $(A,\lambda,\i,\ebar)$, and let $*$ denote the Rosati-involution on the $\q$-algebra 
$\End_B^0(A)$. Let us write $I_x/\q$ for the group scheme which represents the functor 
\begin{equation}
C\mapsto\{g\in(\End_B^0(A)\otimes C)^\times|gg^t\in C^\times\}.
\end{equation}
Every full level structure $\eta:V\otimes\adp\cong H_1(A,\adp)$ yields an
isomorphism 
$$I\times\adp\stackrel{\cong}{\rightarrow}G\times\adp;\gamma\mapsto\eta^{-1}\gamma\eta.$$
Notice that the preimage of $K^p$ under the above isomorphism depends only on the
$K^p$-orbit of $\eta$, and hence we can define $K_x^p:=\eta K^p\eta^{-1}$ for
any $\eta\in\ebar$, this is again a compact open subgroup of $I_x(\adp)$.
Consider the compact set $\tilde K_p:=\{\gamma\in I(\q_p)|\gamma,\gamma^{-1}\in p^{-1}\z_p\otimes\End_B(A)\}$,
and let us say that $K^p$ is superneat for $x$ if and only if 
$I_x(\q)\cap\tilde K_p\times K_x^p=\{1\}$.
The left-hand side is always a finite group, because $I_x$ is anisotropic. In particular $K^p$
will always contain some a compact open subgroup which is superneat for every $x\in S_{K^p}^{si}$

\begin{lem}
If $K^p$ is superneat for $x$, then \eqref{para} is a closed immersion.
\end{lem}
\begin{proof}
A morphism from a proper $\f_p^{ac}$-variety to a separated one is a closed immersion if and only if it radicial 
and injective on the tangent spaces to all $\f_p^{ac}$-valued points, this is elementary and can be proved 
along the lines of \cite[Lemma 7.4.]{hart}. In view of lemma \ref{immersed} it suffices to check that \eqref{para} is 
indeed injective on geometric points. Suppose it wasn't. Then there existed 
$S_{K^p}(k)\ni x_1=(A_1,\lambda_1,\i_1,\ebar_1)$ which lies in the image of \eqref{para} in two 
different ways. According to the thoughts at the end of subsection \ref{Fermat}, this means that there 
existed two degree-$p^8$-isogenies $e,e':A_1\rightarrow A$ each of which induce the additional structures 
$\lambda_1$, $\i_1$, $\ebar_1$ from the additional structures $\lambda$, $\i$, $\ebar$ on $A$. It follows 
immediately that $\id_A\neq e'\circ e^{-1}$ is in contradiction to $K^p$ being superneat for $x$.
\end{proof}

\end{document}